\documentclass[12pt,reqno]{amsart}

\usepackage{amsfonts,amsmath}
\usepackage{amssymb,amsthm}		
\usepackage{amsthm}
\usepackage[T1]{fontenc}
\usepackage{lmodern}
\usepackage{hyperref}
\usepackage[all]{hypcap}
\usepackage{fullpage}
\usepackage{color}
\usepackage{graphicx}
\usepackage[dvipsnames]{xcolor}
\usepackage{fancyhdr}
\usepackage{psfrag}
\usepackage{enumerate}
\usepackage{color}
\usepackage{verbatim}

\usepackage{indentfirst}
\usepackage{float}
\usepackage{latexsym}
\usepackage{graphics}
\usepackage{verbatim}

\newcommand{\Q}{{\mathbb Q}}

\newcommand{\Qbar}{{\overline{\Q}}}

\newcommand{\Z}{{\mathbb Z}}

\newcommand{\FF}{{\mathbb F}}

\newcommand{\Gal}{{\rm Gal}}
\newcommand{\GL}{{\rm GL}}

\newcommand{\ndivides}{%
  \mathrel{\mkern.5mu 
    \ooalign{\hidewidth$\big|$\hidewidth\cr$\nmid$\cr}%
  }%
}

\newtheorem{theorem}{Theorem}
\newtheorem*{theorem*}{Theorem}
\newtheorem{lemma}[theorem]{Lemma}
\newtheorem{corollary}{Corollary}
\newtheorem{conjecture}{Conjecture}

\newtheorem{remark}{Remark}
\newtheorem*{acknowledgement}{Acknowledgement}

\DeclareFontFamily{T1}{calligra}{}
\DeclareFontShape{T1}{calligra}{m}{n}{<->s*[1.44]callig15}{}

\begin{document}

\author{S. Baier}
\address{R. K. M. Vivekananda University, Department of Mathematics, P. O. Belurmath, District Howrah, WB, India 711202 \\
email: stephanbaier2017@gmail.com}
\author{Vijay M. Patankar}
\address{School of Physical Sciences, Jawaharlal Nehru University, New
Delhi, INDIA 110067 \\
email: vijaypatankar@gmail.com}

\title{Applications of the square sieve to a conjecture of Lang and Trotter for a pair of elliptic curves over the rationals}

\date{\today}

\begin{abstract} 
Let \( E \) be an elliptic curve over \( \Q \). Let \( p \) be a prime of good reduction for \( E \). Then, for a prime \( p \neq \ell \), the 
Frobenius automorphism associated to \( p \) (unique up to conjugation) acts on the \( \ell \)-adic Tate module of \( E \). The characteristic polynomial of the Frobenius 
automorphism has rational integer coefficients and is independent of \( \ell \). Its splitting field is called the Frobenius field of \( E \) at \( p \). Let \( E_1 \) 
and \( E_2 \) be two elliptic curves defined over \( \Q \) that are non-isogenous over \( \Qbar \) and also without complex multiplication over \( \Qbar \). In analogy with the well-known
Lang-Trotter conjecture for a single elliptic curve, it is natural to consider the asymptotic behaviour 
of the function that counts the number of primes \( p \leq x \) such that the 
Frobenius fields of \( E_1 \) and \( E_2 \) at \( p \) coincide.  In this short note, using Heath-Brown's square sieve, we provide both conditional (upon the Generalized Riemann Hypothesis) and unconditional upper bounds.   
\end{abstract}

\maketitle

\begin{center} 
{\fontencoding{T1}\fontfamily{calligra}\selectfont To V. Kumar Murty:  On the occasion of his sixtieth birthday}
\end{center}

\section{Introduction and Statement of the main theorem}

In \cite{LT}, Lang and Trotter state their conjectures about the distribution of images of Frobenius automorphisms in 
\( {\rm GL}_2 \) extensions of \( \Q \). These conjectures can be explicitly stated in the context of elliptic curves and their associated Galois representations. A 
generalization of these
conjectures in the setting of a strictly compatible system of $\lambda$-adic Galois representations over a number field was developed by
V. Kumar Murty in \cite{KM}. 

Let \( E \) be an elliptic curve defined over the field of rational numbers with conductor \( N \). Let \( G_\Q \) denote the absolute Galois 
group of \( \Q \), i.e. the Galois group of \( \Qbar \) over \( \Q \). For a prime \( \ell \), let \( \rho_\ell \) be the associated representation of \( G_\Q \)
into \( {\rm GL}_2 ( \Z_\ell ) \) that represents the action of \( G_\Q \) on the \( \ell \)-adic Tate module attached to \( E \). Then, a prime 
\( p \nmid  N \ell \) is unramified in the subfield of \( \Qbar \) defined as the fixed field of \( {\rm Ker} ( \rho_\ell ) \). Thus, the Frobenius automorphism attached to \( p \) in \( G_\Q / {\rm Ker} (
 \rho_\ell ) \) is well defined up to conjugation. We denote the associated conjugacy class by \( \sigma_p \). As a consequence, the characteristic polynomial of \( \rho_
\ell ( \sigma_p ) \), say, \( \phi_p (x) := x^2 - t_p x + p \), is well-defined. It is known that the characteristic polynomial has rational integer coefficients, is independent of \( \ell \), and \( t_p = 1 + p - \# E ( \FF_p ) \), where \( \# E ( \FF_p ) \) is the number of \( \FF_p \)-rational points of the reduction of \( E \) modulo \( p \). Furthermore, \( \phi_p (x) \) has complex conjugate roots of absolute value \( \sqrt{p} \). 
We then define the Frobenius field associated to \( E \) at \( p \) as 
\( F(E, p ) : = \Q  \left(\sqrt{t_p^2-4p}\right) \). Note that \( F(E, p ) \) is the splitting field of \( \phi_p (x ) \).

We now state the precise forms of the Lang and Trotter conjectures for a single elliptic curve \( E \) (with notations as above). The details about the explicit (product) formula for \( C (E, t) \) in Conjecture \ref{LT-trace} can be found in \cite[section 4]{LT} -- especially \cite[page 33]{LT} and \cite[Theorem 4.2, page 36]{LT}. 

\begin{conjecture} \label{LT-trace}{\rm (Fixed trace)}
Let \( E \) be an elliptic curve defined over the rational numbers and without complex multiplication. Let \( t \) be an integer. Let 
\[ 
S(E, t; x) :=  \#  \left\{ p \le x\ :\ p \ndivides N ,  ~~~ t_p = t \right\} .
\] 
Then, there exists a constant \( C ( E, t) \) such that
\[
S(E, t; x ) \sim C(E,t) \frac{\sqrt{x}}{{\rm log} ~x}
\] 
as \( x \) tends to \( + \infty \). 
\end{conjecture}

\begin{remark}
In the above, \( C(E, 0 ) >  0 \). On the other hand, when \( t \neq 0 \) it may happen that \( C(E,t) = 0 \). This is because the presence of non-trivial 
torsion points in \( E ( \Q ) \) may impose congruence conditions on the traces of the Frobenius automorphism. The details can be found in \cite[Remark 1, page 37]{LT}.
When \( C(E, t) = 0 \), the conjecture will be interpreted to mean that \(S(E, t; x) \) is a bounded function.
\end{remark}  
 
 \begin{conjecture}\label{LT-Field}{\rm (Fixed Frobenius field)}
Let \( E \) be an elliptic curve defined over the rational numbers and without complex multiplication.  
Let \( F \) be an imaginary quadratic field. Let 
\[
S(E, F; x) :=  \#  \left\{ p \le x\ :\ p \ndivides N ,  ~~~ F(E,p) = F \right\} .
\] 
Then, there exists a constant \( C ( E, F) > 0 \) such that
\[
S(E, F; x ) \sim C(E,F) \frac{\sqrt{x}}{{\rm log} ~x}
\] 
as \( x \) tends to \( + \infty \).
\end{conjecture}

The details about the above conjecture can be found on page 69, page 109, Theorem 5.1 on page 
110, and Theorem 6.3 on page 116 of \cite{LT}.  

Based on heuristic arguments as explained in \cite[Remark 2, pages 37 -- 38]{LT}, Lang and Trotter make the following conjecture regarding coincidence of  supersingular primes for a pair of 
elliptic curves. Let $E_1$ and $E_2$ be two elliptic curves over $\Q$ both without complex 
multiplication and with conductors $N_1$ 
and $N_2$, respectively. For \( p \nmid N_1 N_2 \), let \(  a_p (E_1) := p + 1 - \# E_1 ( \FF_p )  \) 
and \( a_p (E_2) := p+1-  \# E_2 ( \FF_p ) \). 

\begin{conjecture}\label{LTsspair} {\rm (Supersingular primes for a pair of elliptic curves)} 
With notation as above, let 
\[
S(E_1, E_2, 0;  x) :=  \#  \left\{ p \le x\ :\ p \ndivides N_1 N_2 ,  a_p (E_1) = a_p (E_2) = 0 \right\} .
\] 
Then, 
\[
S(E_1, E_2, 0; x )  = O  ( \log \log x ) 
\]
if and only if 
\( E_1 \) and \( E_2 \) are non-isogenous over  \( \Qbar \). 
\end{conjecture}

\begin{remark}
Suppose \( E_1 \) and \( E_2 \) are isogenous over \( \Q \). Then, by Conjecture \ref{LT-trace}, and the fact that \( C(E,0) > 0 \), it follows that \( S ( E_1 , E_2 , 0; x ) \sim C(E,0 ) \frac{\sqrt{x}}{\log x} \neq O  ( \log \log x ) \). 
\end{remark}

Since the submission of this paper and the previous \cite{KPR}, an explicit version of a more 
general conjecture has been proposed by Akbary and Park in \cite{AP}. 

\begin{conjecture} \cite[Conjecture 1.2]{AP} {\rm (Pair of elliptic curves and pair of fixed traces)}
Let \( E_1 \) and \( E_2 \) be two non-isogenous elliptic curves defined over \( \Q \) with conductors \( N_1 \) and \( N_2 \) respectively and both without complex multiplication. For fixed integers \( t_1 \) and \( t_2 \),  let 
\[
S(E_1,E_ 2, t_1,t_2; x) := \# \{ p \leq x  ~:   ~ p \nmid N_1 N_2, ~a_p (E_1) = t_1 ~{\rm and~} a_p (E_2) = t_2 \} .
\]
Then, there exists a constant \( C(E_1,E_2,t_1,t_2 ) \geq 0 \) such that
\[
S(E_1,E_ 2, t_1,t_2; x) \sim  C (E_1,E_2,t_1,t_2 ) \log\log x 
\]
as \( x \rightarrow \infty \). 
\end{conjecture}
The conjectural constant \( C ( E_1, E_2, t_1, t_2 ) \) is the focus of their paper. Thus, setting \( t_1 = t_2 = 0 \), one can recover a precise form of Conjecture \ref{LTsspair}.

Similarly, we would like to study the function that counts the number of primes of good 
reduction $p\le x$ such that the corresponding Frobenius 
fields for $E_1$ and $E_2$ are equal, i.e., 
\begin{equation} \label{SxE1E2}
S(E_1 , E_2; x ):= \# \left\{ p \le x\ :\ p \ndivides N_1 N_2 ,  ~~~~~~~~ F( E_1 , p ) = F (E_2 , p )  \right\}.
\end{equation}

Let us note that if \( E \) is an elliptic curve over \( \Q \) without complex multiplication. Then, the set of Frobenius fields 
\( F(E, p) \) as \( p \) runs over the primes of good reduction for \( E \) is an infinite set. This follows from a series of exercises in Serre's book \cite[
Chapter IV, pages 13--14]{Se1}.

In this context, we quote the last sentence on page 38 of \cite[Remark 2]{LT}: ``\emph{Of course a similar conjecture can be made about primes whose Frobenius elements 
for two given curves both generate the same quadratic 
field}.'' This statement is not very explicit but it seems to indicate a plausible conjectural answer to the asymptotic behaviour of 
\( S(E_1 , E_2; x ) \). 

\begin{conjecture}\footnote{Unfortunately and inadvertently, this conjecture was wrongly stated in \cite[Conjecture 1]{KPR}.} \cite[Conjecture 1]{KPR}\label{LTpairequalFF} {\rm (Pair of elliptic curves and equal Frobenius fields)} Let \( E_1 \) and \( E_2 \) be two elliptic curves over the
rationals, and both without complex multiplication over \( \Qbar \). Then,
\( E_1 \) is not isogenous to \( E_2 \) over \( \Qbar \) if and only if
\[
S(E_1 , E_2; x ) = O ( \log \log x ).
\]
\end{conjecture}

\begin{remark}
Suppose \( E_1 \) and \( E_2 \) are isogenous over \( \Qbar \), hence over some number field \( L \). By extending \( L \) if necessary, we can assume 
that \( L \) is Galois over \( \Q \). Then, for any prime \( p \nmid N_1 N_2 \) and that splits in \( L \), \( a_p (E_1 ) = a_p (E_2) \). 
Thus, from the Chebotarev density theorem, it follows that \( S( E_1 , E_2 ; x ) \) grows at least as much as \( \frac{1}{[L:\Q]} \frac{x}{\log x} \). 
Hence, \( S(E_1 , E_2; x ) \neq O  ( \log\log x ) \). 
\end{remark}


Using techniques from \( \ell \)-adic representations, the following result was proved in \cite{KPR}.

\begin{theorem} \cite[Theorem 3]{KPR} \label{Mult-one-FF} Let \( E_1 \) and \( E_2 \) be two
elliptic curves over a number field \( K  \). Let $\Sigma_r$ be a finite
subset of the set  \( \Sigma_K \) of finite places of  \(K\)
containing the places of bad reduction of $E_1$ and $E_2$.  Assume
that at least one of the elliptic curves is without complex
multiplication. Let
\[ 
\mathcal{S} (E_1, E_2 ) := \{ v \in \Sigma_K ~\backslash ~\Sigma_r~\mid ~ F (E_1, v) = F ( E_2, v) \}.
\] 
Then, \( E_1 \) and \( E_2 \) are isogenous over a finite extension of \( K \) if and only if 
\( \mathcal{S} ( E_1 , E_2 ) \) has  positive upper density.
\end{theorem}
In the above, the upper density of a subset of primes \( S \) is defined as
\[ 
ud(S) := {\limsup}_{x \rightarrow \infty} \frac{ \# \{ v \in \Sigma_K ~\backslash ~\Sigma_r~:~ Nv \leq x \textrm{~and~} v \in S \} }{\# \{ v \in \Sigma_K ~:~ Nv \leq x \} },
\]
where \( Nv \) denotes the cardinality of the reside field of \( K \) at \( v \).
\medskip\par
\begin{remark}
If \( p \) is a common supersingular prime, then \( F(E_1, p ) = F( E_2, p ) = \Q ( \sqrt{ - p } ) \).  
Thus, Conjecture \ref{LTsspair} and Conjecture \ref{LTpairequalFF} are entangled. 
\end{remark} 

In this note, we establish non-trivial upper bounds on \( S(E_1, E_2; x) \), both conditional and unconditional. Our results are stated below. 

Assuming the Generalized Riemann Hypothesis (GRH) for the Dedekind zeta functions of number fields, we are able to prove the following assertion.
\begin{theorem} \label{thmconditional}
Let \( E_1 \) and \( E_2 \) be two elliptic curves defined over \( \Q \). Suppose \( E_1 \) and \( E_2 \) are non-isogenous over \( \Qbar \), and also both without complex multiplication over \( \Qbar \). 
Then, under GRH, we have
\[
S(E_1 , E_2; x )\ll x^{29/30}(\log x)^{1/15}, 
\]
where the implied constant depends only on \( E_1 \) and \( E_2 \).
\end{theorem}
Without assuming the GRH, we prove the following result. 
\begin{theorem} \label{thmunconditional}
Let \( E_1 \) and \( E_2 \) be two elliptic curves defined over \( \Q \). 
Suppose \( E_1 \) and \( E_2 \) are non-isogenous over \( \Qbar \), and 
also both without complex multiplication over \( \Qbar \).  Then, 
\[
S(E_1 , E_2; x)\ll \frac{x (\log \log x)^{22/21}}{(\log x)^{43/42}}.
\]
\end{theorem}

\begin{remark}
Note that Theorem \ref{thmunconditional} implies Theorem \ref{Mult-one-FF}
by the Prime Number Theorem under the additional assumption that both the 
elliptic curves are without complex multiplication.
\end{remark}

\begin{remark}
The results as stated above can be extended to a pair of elliptic curves 
without complex multiplication defined over any number field. 
\end{remark}

\begin{remark}
It is mentioned in the last paragraph of \cite[Section 6, pages 1174--1175]{CFM} that as an application of Theorem 10 of \cite{Se3} one can 
obtain the following estimate under the GRH.
\[
S( E_1, E_2; x ) = O_{{E_1, E_2}} ( x^{\frac{11}{12}} ),
\]
where the \( O \)-constant depends in an unspecified way on \( E_1 \) and \( E_2 \). This is a better bound than the one 
obtained in Theorem \ref{thmconditional}, but here the implied $O$-constant depends in an unspecified way on \( E_1 \) and \( E_2 \), whereas 
in Theorems \ref{thmconditional} and \ref{thmunconditional}, the implied constants are effectively computable. Moreover, our method is different
and a rather easy application of the square sieve and the Chebotarev density theorem.
\end{remark}

\section{Square-sieve approach}

Our technique closely follows that of \cite{CFM}, where the authors apply Heath-Brown's square sieve to bound the function that 
counts the number of primes \( p \leq x \) such that the associated Frobenius field at \(  p \) equals the given imaginary quadratic field.

Clearly, $S(E_1,E_2;x)$, as defined in \eqref{SxE1E2}, equals the number of primes $p\le x$ with $p \ndivides N$ such that  
\begin{equation}
\begin{split}
4p-a_p^2 = & Dm^2\\
4p-b_p^2 = & Dn^2
\end{split}
\end{equation}
for some square-free $D$ and natural numbers $m$ and $n$, where \( a_p := a_p (E_1) \) and \( b_p := a_p (E_2) \).  
For a given $p$, it is easy to see that the above system is satisfied for a squarefree $D$ and $m,n$ natural numbers if and only if 
$(4p-a_p^2)(4p-b_p^2)$ is a square. 
To see this, we note that if 
\begin{equation}
\begin{split}
4p-a_p^2 = & D_1m^2\\
4p-b_p^2 = & D_2n^2
\end{split}
\end{equation}
with $D_1$ and $D_2$ square-free, then 
$$
\left(4p-a_p^2\right)\left(4p-b_p^2\right)=D_1D_2m^2n^2
$$
is a square if and only of $D_1=D_2$. 
Hence,
$$
S(E_1,E_2;x) =\#  \left\{ p \le x\ :\ p \ndivides N, \ (4p-a_p^2)(4p-b_p^2)=\Box\right\},
$$
where $N=N_1N_2$ is the product of the conductors $N_1$ and $N_2$ of $E_1$ and $E_2$, respectively. 
We detect squares in the multiset
\begin{equation} \label{ourmulti}
\mathcal{A}:=\{(4p-a_p^2)(4p-b_p^2)\ :\  p\le x, \ p \ndivides N \}
\end{equation}
by using a version of Heath-Brown's square sieve. The version 1 below, which is a consequence  of Theorem 1 in \cite{HB}, is best suited for obtaining a conditional bound under
GRH. 
The version 2, as stated in section \ref{unconditional}, is better suited for an unconditional version
(without assuming GRH) of the theorem. 

\begin{lemma}[square sieve - version 1] \label{squaresieve}
Let $\mathcal{A}$ be a multiset of positive integers, 
and let $\mathcal{P}$ be a set of $P$ distinct primes. Set
\[
S(\mathcal{A}) := \# \{ \alpha \in  \mathcal{A}\ :\ \alpha \mbox{ is a square}\}
\]
and suppose that
\begin{equation} \label{uppercond}
\max\limits_{n\in \mathcal{A}} n \le e^P.
\end{equation}
Then
\begin{equation} \label{oldS}
\# S(\mathcal{A}) \ll \frac{\# \mathcal{A}}{P} +\frac{1}{P^2}\cdot  
\sum\limits_{\substack{q_1,q_2\in \mathcal{P}\\ q_1\not=q_2}} \left|\sum\limits_{n \in \mathcal{A}} \left(\frac{n}{q_1q_2}\right)\right|,
\end{equation}
where $\left(\frac{n}{q_1q_2}\right)$ is the Jacobi symbol.
\end{lemma}

We now apply this lemma with $\mathcal{A}$ as in \eqref{ourmulti},
\begin{equation} \label{Pdef}
 \mathcal{P} := \{ q\in \Sigma_\Q  : \ z/2< q \le z \} \quad \mbox{and} \quad P:=\#\mathcal{P},
 \end{equation}
where $\Sigma_\Q$ denotes the set of primes of \( \Q \) and $z$ is a positive parameter satisfying
 \begin{equation} \label{zcond}
 e^P=e^{\pi(z)-\pi(z/2)}>x
 \end{equation}
 which will be suitably chosen later. It follows that 
\begin{equation} \label{basicbound}
S(E_1,E_2;x)\ll 1+\frac{1}{P} \cdot  (\pi(x)-\omega(N)) +\frac{1}{P^2} \cdot \sum\limits_{\substack{q_1,q_2\in \mathcal{P}\\ q_1\not=q_2}} \left|\sum\limits_{\substack{p\le x\\
 p\not\ \! | q_1q_2N}} \left(\frac{(4p-a_p^2)(4p-b_p^2)}{q_1q_2}\right)\right|,
 \end{equation}
where $\omega(N)$ is the numbers of distinct prime factors of $N$. 
 Moreover, by the prime number theorem,
 \begin{equation} \label{primenumber}
 P\sim \frac{z}{2\log z} \quad \mbox{and} \quad \pi(x)-\omega(N) \sim \frac{x}{\log x}.
 \end{equation}
 Thus, to ensure that \eqref{zcond} is satisfied for large enough $x$, it suffices that
 \begin{equation} \label{zcond1}
 z> (\log x)^{1+\varepsilon},
 \end{equation}
 which we assume from now on. 
 
 We now estimate the character sum in \eqref{basicbound}. We break the summation into residue classes as follows.
 \begin{equation} \label{red}
 \sum\limits_{\substack{p\le x\\
 p\not\ \! | q_1q_2N}} \left(\frac{(4p-a_p^2)(4p-b_p^2)}{q_1q_2}\right) = \sum\limits_{\substack{d=1\\ (d,q_1q_2)=1}}^{q_1q_2} \sum\limits_{s=1}^{q_1q_2}
 \sum\limits_{t=1}^{q_1q_2}  \left(\frac{(4d-s^2)(4d-t^2)}{q_1q_2}\right) \cdot \pi\left(x; d,s,t\right),
 \end{equation}
where
\begin{equation} 
\begin{split}
\pi(x;d,s,t):=\#\{p\le x\ :\  & p  \ndivides N, \ p\equiv d\bmod{q_1q_2}, \\ &  a_p\equiv s\bmod{q_1q_2}, \ b_p \equiv t\bmod{q_1q_2}\}.  
\end{split}
\end{equation}
The task is now to evaluate $\pi(x;d,s,t)$ asymptotically. 

We first recall some group theoretical results from \cite{CFM}. We will then apply versions of 
the Chebotarev density theorem (under GRH or without) and obtain our results. 

\section{Application of Chebotarev density theorem under the GRH} \label{applichebo}
For $i=1,2$ let $E_i[k]$ be the group of $k$-division points 
of $E_i$ and $F_k^i:=\mathbb{Q}(E_i[k])$ be the field obtained by adjoining to $\mathbb{Q}$ the $x$ and the $y$ coordinates of $k$-division points. Let
\begin{equation} \label{iso}
G_k^i:=\mbox{Gal}\left(F_k^i / \mathbb{Q}\right).
\end{equation}
The following facts were used in \cite{CFM} and are of importance here as well. 
The action of \( \Gal ( \Qbar / \Q ) \) on the \( k \)-division points on \( E_i \) gives rise to the natural Galois representations 
\begin{equation} \label{galrep}
	\phi_k^{i}\ :\ G_k^i \rightarrow \GL_2(\mathbb{Z}/k\mathbb{Z}) .
\end{equation}
which are injective for every $k\in \mathbb{N} $. By \cite[Theorem 2 on page 37 and Theorem 3 on page 42]{Se2}, it follows that for elliptic curves \( E_i \) 
without complex multiplication, there exist constants \( C(E_i) \) depending on \( E_i \) such that \( \phi_k^{i} \) are surjective whenever 
\( (k, C(E_i))=1 \). 

In our applications, $k=q_1q_2$, and so the condition $(q_1q_2, C(E_i))=1$
will be satisfied for $i=1,2$ if $x$ is large enough since we have 
assumed that
$z/2 < q_1,q_2\le z$ and $z > (\log x)^{1+\varepsilon}$. Hence, we have 
\[
G_{q_1q_2}^i \cong  \GL_2(\mathbb{Z}/q_1q_2\mathbb{Z}),   ~~~~~~ \mbox{if} ~x ~\textrm{is large enough}.
\]
Further, $\phi_k^i$ has the properties that 
\[
\mbox{\rm tr}(\phi_k^1(\sigma_p^1))=a_p, \quad \mbox{\rm tr}(\phi_k^2(\sigma_p^2))=b_p, \quad \mbox{\rm det}(\phi_k^1(\sigma_p^1))=p=
\mbox{\rm det}\left(\phi_k^2(\sigma_p^2)\right),
\]
where $\sigma_p^i$ is the Frobenius conjugacy class associated to $p$ in $\Gal ( F_k^i / \Q  ) $.
Equipped with the above, we are now ready for an application of the Chebotarev density theorem under GRH (see \cite[Theorem 2.3]{CFM}). 

\begin{theorem}[Chebotarev density theorem under GRH] \label{chebo1} 
Let $L $ be a finite Galois extension of $\Q$ with Galois group $G$. Let $n_L$ be the degree of $L$ over $\Q$ and $d_L$ its discriminant. Let $C$ be a union of conjugacy classes of 
$G$. Define 
\[
\pi_C (x, L, \mathbb{Q}) := \# \{p \le x\ : \ p \mbox{ unramified in } L / \Q,\  \sigma_{p} \subseteq C \},
\]
where $\sigma_p$ is the Frobenius conjugacy class associated to \( p \) in the extension $L / \Q $. Then, under the GRH for Dedekind zeta functions, 
we have
$$
\pi_C(x,L, \mathbb{Q}) = \frac{\# C}{\# G}\cdot \mbox{\rm li}(x) + 
O\left(\# C \cdot x^{1/2}\left(\log x+ \frac{\log d_L}{n_L}\right)\right),
$$
where the implicit constants are absolute. 
\end{theorem}

We wish to apply the above theorem for  $L:=F_{q_1q_2}^1F_{q_1q_2}^2$. Let
\( H_{q_1q_2} \) be the Galois group \( \Gal ( L / \Q ) \).  
We now use Serre's results on the image of \( \Gal ( \Qbar / \Q ) \) under the 
product of \( \ell \)-adic representations attached to two non-isogenenous 
elliptic curves (both without complex multiplication). In fact, \cite[Corollary 2, page 324]{Se2} implies that   
\begin{equation*}
H_{q_1q_2} = \left\{  (g_1,g_2) \in G_{q_1q_2}^1 \times G_{q_1q_2}^2  \ :\   \mbox{det} ~ \phi_{q_1q_2}^1 (g_1) = \mbox{det} ~ \phi_{q_1q_2}^2 (g_2) \right\},
\end{equation*}
with $\phi_{q_1q_2}^i$ $(i=1,2)$ as in \eqref{galrep}.
Let \( C \) be the following conjugacy class in \( G = H_{q_1q_2}  \). 
\begin{equation*}
\begin{split}
C := C_{q_1q_2}(s,t,d):=\big\{(g_1,g_2)\in G_{q_1q_2}^1\times G_{q_1q_2}^2\ :\  &  \mbox{det } \phi_{q_1q_2}^1(g_1)=d= \mbox{det }\phi_{q_1q_2}^2(g_2), \\  &  
\mbox{tr } \phi_{q_1q_2}^1(g_1)=s, \ \mbox{tr } \phi_{q_1q_2}^2(g_2)=t\big\}.
\end{split}
\end{equation*}
Applying Theorem \ref{chebo1}, we get
\begin{equation} \label{CheboGRH}
\pi(x;d,s,t)=\frac{\# C_{q_1q_2}(s,t,d)}{\# H_{q_1q_2}} \cdot \mbox{li}(x) + 
O\left(\# C_{q_1q_2}(s,t,d) x^{1/2}\left(\log x + \frac{\log d_L}{n_L}\right)\right)
\end{equation}
under GRH. 

Combining \eqref{red} and \eqref{CheboGRH}, and taking into account that
\begin{equation} \label{sum}
\sum\limits_{\substack{d=1\\ (d,q_1q_2)=1}}^{q_1q_2} \sum\limits_{s=1}^{q_1q_2}
\sum\limits_{t=1}^{q_1q_2} \# C_{q_1q_2}(s,t,d) = \# H_{q_1q_2},
\end{equation}
we have
\begin{equation} \label{appli}
\begin{split}
& \sum\limits_{\substack{p\le x\\
p  \ndivides q_1q_2N}} \left(\frac{(4d-a_p^2)(4d-b_p^2)}{q_1q_2}\right) \\ = & \mbox{li}(x)\cdot 
\sum\limits_{\substack{d=1\\ (d,q_1q_2)=1}}^{q_1q_2} \sum\limits_{s=1}^{q_1q_2}
\sum\limits_{t=1}^{q_1q_2}  \left(\frac{(4d-s^2)(4d-t^2)}{q_1q_2}\right) \cdot 
\frac{\# C_{q_1q_2}(s,t,d)}{\# H_{q_1q_2}} +\\ & O\left(\# H_{q_1q_2}x^{1/2}
\left(\log x +\frac{\log d_L}{n_L}\right)\right).
\end{split}
\end{equation}

\section{Counting} \label{count}
We now compute \( \# C_{q_1q_2}(s,t,d) \) and \( \# H_{q_1q_2} \) that appear on the right hand side of equation \eqref{appli}.

Since, for $q_1,q_2$ large enough, 
\[
\phi_{q_1q_2}^i\ : \ F_{q_1q_2}^i \rightarrow \mbox{GL}_2\left(\mathbb{Z}/ q_1q_2\mathbb{Z}\right), \quad i=1,2 
\]
are isomorphisms, we have 
\begin{equation} \label{noC}
\begin{split}
\# C_{q_1q_2}(s,t,d)=\# \big\{(A_1,A_2)\in \mbox{GL}_2\left(\mathbb{Z}/ q_1q_2\mathbb{Z}\right)^2 \ :\   & \mbox{det} (A_1)=d= \mbox{det} (A_2), \\ &
\mbox{tr}(A_1)=s, \ \mbox{tr}(A_2)=t\big\}
\end{split}
\end{equation}
and 
\begin{equation} \label{noH}
\# H_{q_1q_2}=\# \big\{(A_1,A_2)\in \mbox{GL}_2\left(\mathbb{Z}/ q_1q_2\mathbb{Z}\right)^2 \ :\   \mbox{det} (A_1)= \mbox{det} (A_2)\big\}.
\end{equation}

We now use Corollary 2.8 from \cite{CFM} which states the following.

\begin{corollary} \label{hungry} Let \(q_1\) and \(q_2\) be two distinct odd primes, and \(d, t \in \Z/q_1q_2\Z\) be fixed with
\((d, q_1q_2) = 1\). Then
\[
\begin{split}
& \#\{g \in {\rm GL}_2(\Z/q_1q_2\Z)\ :\ \mbox{\rm det}(g) = d, \ \mbox{\rm tr}(g) = t\}\\ = & q_1q_2\left(q_1+\left(\frac{t^2-4d}{q_1}\right)\right)
\left(q_2+\left(\frac{t^2-4d}{q_2}\right)\right).
\end{split}
\]
\end{corollary}

Applying Corollary \ref{hungry}, we deduce that
\begin{equation*}
\begin{split}
\# C_{q_1q_2}(s,t,d)=q_1^2q_2^2 & \left(q_1+\left(\frac{s^2-4d}{q_1}\right)\right)\left(q_2+\left(\frac{s^2-4d}{q_2}\right)\right)\times\\ &
\left(q_1+\left(\frac{t^2-4d}{q_1}\right)\right)\left(q_2+\left(\frac{t^2-4d}{q_2}\right)\right).
\end{split}
\end{equation*}
Since $\big\{(A_1,A_2)\in \mbox{GL}_2\left(\mathbb{Z}/ q_1q_2\mathbb{Z}\right)^2 \ :\   \mbox{det} (A_1)= \mbox{det} (A_2)\big\}$ is the kernel of the epimorphism
$$
\sigma\ :\ \mbox{GL}_2\left(\mathbb{Z}/ q_1q_2\mathbb{Z}\right)^2 \rightarrow (\mathbb{Z}/q_1q_2\mathbb{Z})^{\ast}
$$
defined by
$$
\sigma(A_1,A_2)= \mbox{det}(A_1)\cdot \mbox{det}(A_2)^{-1},
$$
it follows that
\begin{equation*}
\begin{split}
\# H_{q_1q_2}= & \frac{\# \mbox{GL}_2\left(\mathbb{Z}/ q_1q_2\mathbb{Z}\right)^2}{\# (\mathbb{Z}/q_1q_2\mathbb{Z})^{\ast}}=
\frac{q_1^2(q_1-1)^2(q_1^2-1)^2q_2^2(q_2-1)^2(q_2^2-1)^2}{(q_1-1)(q_2-1)}\\
= & q_1^2(q_1-1)(q_1^2-1)^2q_2^2(q_2-1)(q_2^2-1)^2.
\end{split}
\end{equation*}
Hence,
\begin{equation} \label{this}
\begin{split}
\frac{\# C_{q_1q_2}(s,t,d)}{\# H_{q_1q_2}}= & \frac{1}{(q_1-1)(q_1^2-1)^2(q_2-1)(q_2^2-1)^2}\times\\ & \left(q_1+\left(\frac{s^2-4d}{q_1}\right)\right)\left(q_2+\left(\frac{s^2-4d}{q_2}\right)\right)\times\\ &
\left(q_1+\left(\frac{t^2-4d}{q_1}\right)\right)\left(q_2+\left(\frac{t^2-4d}{q_2}\right)\right)\\ = & 
\frac{q_1^2q_2^2}{(q_1-1)(q_1^2-1)^2(q_2-1)(q_2^2-1)^2}+O\left(\frac{1}{z^7}\right)
\end{split}
\end{equation}
if $z/2<q_1,q_2\le z$. This allows us to express the main term on the right-hand side of \eqref{appli} explicitly. 
For the estimations of the $O$-term in \eqref{appli} and an $O$-term occuring later in this paper, we 
prove the following estimates for $n_L$, the degree of $L / \Q $, and 
$d_L$, the discriminant of $L / \Q$.

\begin{lemma} \label{ndL}
Suppose that $z/2<q_1,q_2\le z$. Then we have 
$$
n_L\le z^{14}, \quad \frac{\log|d_L|}{n_L}\ll \log z,
$$
$$
\quad |d_L|^{1/n_L}\ll z^{16} \quad  \mbox{and} \quad
\log|d_L| \ll z^{14}\log z.
$$
\end{lemma}

\begin{proof}
Clearly,
\begin{equation} \label{lognL}
n_L= \# H_{q_1q_2}= q_1^2(q_1-1)(q_1^2-1)^2q_2^2(q_2-1)(q_2^2-1)^2\ll z^{14}.
\end{equation}
By a lemma of Hensel (see Lemma 2.6 in \cite{CFM}), we have
\[
\log |d_L| \le n_L\log n_L+(n_L-1)\sum\limits_{p\ \mbox{\scriptsize ramified in } L} \log p
\]
which implies that
\begin{equation} \label{NOS}
\begin{split}
\frac{\log|d_L|}{n_L}\le \log n_L + \log(q_1q_2N) \le 8\log(q_1q_2)+\log(N), 
\end{split}
\end{equation}
where we use the fact that if $p$ ramifies in $L$, then $p|(q_1q_2N)$, which follows from the N\'eron-Ogg-Shafarevich
criterion. From \eqref{NOS}, we deduce that
$$
|d_L|^{1/n_L}= \exp\left(\frac{\log|d_L|}{n_L}\right)\le N(q_1q_2)^8\ll z^{16},
$$
and from \eqref{lognL} and \eqref{NOS}, we deduce that
$$
\log|d_L| \le n_L (\log n_L + \log(q_1q_2N)) \ll z^{14}\log z.
$$
\end{proof}

Plugging \eqref{this} into \eqref{appli} and using Lemma \ref{ndL} 
gives
\begin{equation} \label{appli2}
\begin{split}
& \sum\limits_{\substack{p\le x\\
p \ndivides q_1q_2N}} \left(\frac{(4d-a_p^2)(4d-b_p^2)}{q_1q_2}\right) \\ = & \mbox{li}(x)\cdot \frac{q_1^2q_2^2}{(q_1-1)(q_1^2-1)^2(q_2-1)(q_2^2-1)^2} \cdot
\sum\limits_{\substack{d=1\\ (d,q_1q_2)=1}}^{q_1q_2} \sum\limits_{s=1}^{q_1q_2}
\sum\limits_{t=1}^{q_1q_2}  \left(\frac{(4d-s^2)(4d-t^2)}{q_1q_2}\right)  +\\ & 
O\left(\frac{\mbox{li}(x)}{z}+x^{1/2}z^{14}\log(xz)\right).
\end{split}
\end{equation}

\section{Evaluation of character sums}
Next, we evaluate the character sums above. We have
\begin{equation} \label{into}
\begin{split} 
& \sum\limits_{\substack{d=1\\ (d,q_1q_2)=1}}^{q_1q_2} \sum\limits_{s=1}^{q_1q_2}
\sum\limits_{t=1}^{q_1q_2}  \left(\frac{(4d-s^2)(4d-t^2)}{q_1q_2}\right) = \sum\limits_{\substack{d=1\\ (d,q_1q_2)=1}}^{q_1q_2} \left(\sum\limits_{u=1}^{q_1q_2}
\left(\frac{4d-u^2}{q_1q_2}\right)\right)^2\\ = & \sum\limits_{\substack{d=1\\ (d,q_1)=1}}^{q_1} \left(\sum\limits_{v=1}^{q_1}
\left(\frac{4d-v^2}{q_1}\right)\right)^2\cdot \sum\limits_{\substack{d=1\\ (d,q_2)=1}}^{q_2} \left(\sum\limits_{w=1}^{q_2}
\left(\frac{4d-w^2}{q_2}\right)\right)^2.
\end{split}
\end{equation}
For an odd prime $q$ with $(d,q)=1$, we may write
\begin{equation*} 
\begin{split}
& \sum\limits_{x \bmod{q}} \left(\frac{4d-x^2}{q}\right) = \frac{1}{2}\cdot \sum\limits_{\substack{y \bmod{q}\\  y \not \equiv 0 \bmod{q}}} \left(1+\left(\frac{y}{q}\right)\right) \left(\frac{4d-y}{q}\right) + \left(\frac{4d}{q}\right)\\
= & \frac{1}{2}\cdot \sum\limits_{\substack{y \bmod{q}\\  y\not\equiv 0 \bmod{q}}} \left(\frac{4d-y}{q}\right) + \frac{1}{2}\cdot \sum\limits_{\substack{y \bmod{q}\\  y\not\equiv 0 \bmod{q}}} \left(\frac{y}{q}\right) \left(\frac{4d-y}{q}\right) + \left(\frac{4d}{q}\right)\\
= & \frac{1}{2}\cdot \sum\limits_{\substack{y \bmod{q}}} \left(\frac{4d-y}{q}\right) + \frac{1}{2}\cdot 
\sum\limits_{\substack{y \bmod{q}}} \left(\frac{y}{q}\right) \left(\frac{4d-y}{q}\right) + \frac{1}{2}\cdot \left(\frac{4d}{q}\right).
\end{split}
\end{equation*}
Using the orthogonality relations for Dirichlet characters, the first sum in the last line equals 0, and hence, the above simplifies into
\begin{equation} \label{cc1}
\sum\limits_{x \bmod{q}} \left(\frac{4d-x^2}{q}\right) 
= \frac{1}{2}\cdot \sum\limits_{\substack{y \bmod{q}}} \left(\frac{y}{q}\right) \left(\frac{4d-y}{q}\right) + 
\frac{1}{2}\cdot \left(\frac{4d}{q}\right).
\end{equation}
Now we reduce the right-hand side to Jacobi sums by writing
\begin{equation} \label{cc2} 
\begin{split}
& \frac{1}{2}\cdot \sum\limits_{\substack{y \bmod{q}}} \left(\frac{y}{q}\right) \left(\frac{4d-y}{q}\right) + 
\frac{1}{2}\cdot \left(\frac{4d}{q}\right)\\
= & \frac{1}{2}\cdot \left(\frac{4d}{q}\right)^2 \cdot \sum\limits_{\substack{y \bmod{q}}} \left(\frac{y\overline{4d}}{q}\right) \left(\frac{1-y\overline{4d}}{q}\right) + 
\frac{1}{2}\cdot \left(\frac{4d}{q}\right)\\
= & \frac{1}{2}\cdot \sum\limits_{\substack{z \bmod{q}}} \left(\frac{z}{q}\right) \left(\frac{1-z}{q}\right)+ 
\frac{1}{2}\cdot \left(\frac{d}{q}\right).
\end{split}
\end{equation} 
The last line is evaluated to be
\begin{equation} \label{cc3}
\frac{1}{2}
\cdot \sum\limits_{\substack{z \bmod{q}}} \left(\frac{z}{q}\right) \left(\frac{1-z}{q}\right)+ \frac{1}{2}\cdot \left(\frac{d}{q}\right)
= \frac{1}{2}\cdot \left(-\left(\frac{-1}{q}\right)+\left(\frac{d}{q}\right)\right) \in \{-1,0,1\},
\end{equation}
where we apply the following well-known evaluation of Jacobi sums defined as 
$$
J(\chi, \lambda):=\sum\limits_{a+b=1} \chi(a)\lambda(b),
$$
for the special case $\lambda=\chi^{-1}=(\cdot/q)$.

\begin{lemma} We have
\[
J(\chi, \chi^{-1})=-\chi(-1).
\]
\end{lemma}

\begin{proof} This is \cite[page 93, Theorem 1(c)]{IrRo}.
\end{proof}

Combining \eqref{cc1}, \eqref{cc2} and \eqref{cc3}, we see that
$$
\sum\limits_{x \bmod{q}} \left(\frac{4d-x^2}{q}\right) \in \{-1,0,1\}.
$$
Plugging this into \eqref{into}, we deduce that
\begin{equation} \label{into2}
\begin{split} 
& \sum\limits_{\substack{d=1\\ (d,q_1q_2)=1}}^{q_1q_2} \sum\limits_{s=1}^{q_1q_2}
\sum\limits_{t=1}^{q_1q_2}  \left(\frac{(4d-s^2)(4d-t^2)}{q_1q_2}\right) \le (q_1-1)(q_2-1). 
\end{split}
\end{equation}

\section{Proof of Theorem \ref{thmconditional} } \label{conditional}

Plugging \eqref{into2} into \eqref{appli}, and using \eqref{this}, we get 
\begin{equation*} 
\begin{split}
& \sum\limits_{\substack{p\le x\\
p\ndivides q_1q_2N}} \left(\frac{(4d-a_p^2)(4d-b_p^2)}{q_1q_2}\right) \\ \le & \mbox{li}(x)\cdot \frac{q_1^2q_2^2}{(q_1^2-1)^2(q_2^2-1)^2} 
 + O\left(\frac{\mbox{li}(x)}{z}+x^{1/2}z^{14}\log(xz)\right)\\
 = & O\left(\frac{\mbox{li}(x)}{z}+x^{1/2}z^{14}\log(xz)\right)
\end{split}
\end{equation*}
if $z/2<q_1,q_2\le z$. Combining this with \eqref{basicbound} and \eqref{primenumber}, we deduce that
\[
S(E_1,E_2;x)\ll \frac{x/\log x}{z/\log z} +x^{1/2}z^{14}\log(xz). 
\]
Choosing $z:= x^{1/30}(\log x)^{-1/15}$, we obtain Theorem \ref{thmconditional}.

\section{Proof of Theorem \ref{thmunconditional} } \label{unconditional}

In the following, we modify the method in order to establish an unconditional upper bound for $S(E_1,E_2;x)$.
To this end, we shall need the following second version of the square sieve and an unconditional effective 
version of the Chebotarev density theorem, stated below. 

\begin{lemma}[Square sieve, version 2] \label{squaresieve2} 
Let $\mathcal{A}$ be a multiset of positive integers, 
and let $\mathcal{P}$ be a set of $P$ distinct primes. Set
$$
S(\mathcal{A}) := \#\{ \alpha \in  \mathcal{A}\ :\ \alpha \mbox{ is a square}\}.
$$
Then
\begin{equation} \label{newS}
S(\mathcal{A}) \le \frac{\# \mathcal{A}}{P}+\max\limits_{\substack{q_1,q_2\in \mathcal{P}\\ q_1 \neq q_2}} \left|\sum\limits_{\alpha\in \mathcal{A}}
\left(\frac{\alpha}{q_1q_2}\right)\right|+\frac{2}{P}\cdot \sum\limits_{\alpha\in \mathcal{A}} \sum\limits_{\substack{q\in \mathcal{P}\\ 
q|\alpha}} 1 + \frac{1}{P^2}\cdot \sum\limits_{\alpha\in \mathcal{A}} \left(\sum\limits_{\substack{q\in \mathcal{P}\\ 
q|\alpha}} 1\right)^2.
\end{equation}
\end{lemma}

\begin{proof} This is Theorem 2.1. in \cite{CFM} and originated in \cite{HB}. \end{proof} 

\begin{theorem}[Chebotarev density theorem, unconditional] \label{uncond} Let the conditions of Theorem \ref{chebo1} be kept. There exist positive constants
$A$, $b$ and $b'$ such that, if
\begin{equation} \label{cond1}
\log x\ge bn_L(\log |d_L|)^2,
\end{equation}
then 
\begin{equation}
\begin{split}
\pi_C(x,L:\mathbb{Q})= & \frac{\# C}{\# G} \cdot \mbox{\rm li}(x) + O\left(\frac{\# C}{\# G} \cdot \mbox{\rm li}\left(
x\cdot \exp\left(-b'\cdot \frac{\log x}{\mbox{max}\left\{|d_L|^{1/n_L},\log |d_L|\right\}}\right)\right)\right)+\\ &
O\left(\left(\# \tilde{C}\right) x\exp\left(-A\sqrt{\frac{\log x}{n_L}}\right)\right),
\end{split}
\end{equation}
where $\tilde{C}$ is the set of conjugacy classes whose union is $C$. 
\end{theorem}

\begin{proof}
This is Theorem 2.4 in \cite{CFM}.
\end{proof}

Lemma \ref{squaresieve2} should be compared to Lemma \ref{squaresieve}. We note that in Lemma \ref{squaresieve2}, the 
condition \eqref{uppercond} is omitted, which shall turn out to be essential for us in order to 
obtain an unconditional bound. The costs of omitting this condition are two extra terms on the right-hand side of 
\eqref{newS}, as compared to \eqref{oldS}. We shall be able to estimate these terms quite easily using some results in 
\cite{CFM}. 

Applying Lemma \ref{squaresieve2} in our situation, we now get
\begin{equation*} 
\begin{split}
 S(E_1,E_2;x)\ll & 1+\frac{1}{P}\cdot  (\pi(x)-\omega(N)) +\\ & \frac{1}{P^{2}} \cdot \sum\limits_{\substack{q_1,q_2\in \mathcal{P}\\ q_1 \neq q_2}} \left|\sum\limits_{\substack{p\le x\\
 p\ndivides q_1q_2N}} \left(\frac{(4p-a_p^2)(4p-b_p^2)}{q_1q_2}\right)\right| +\\
 & \frac{2}{P}\cdot \sum\limits_{p\le x} \sum\limits_{\substack{q\in \mathcal{P}\\ 
q|(4p-a_p^2)(4p-b_p^2)}} 1 + \frac{1}{P^2}\cdot \sum\limits_{p\le x} \left(\sum\limits_{\substack{q\in \mathcal{P}\\ 
q|(4p-a_p^2)(4p-b_p^2)}} 1\right)^2
 \end{split}
 \end{equation*}
with $P$ as defined in \eqref{Pdef}.  
To estimate the last two terms on the right-hand side of \eqref{newbound}, we write
$$
\sum\limits_{p\le x} \sum\limits_{\substack{q\in \mathcal{P}\\ q|(4p-a_p^2)(4p-b_p^2)}} 1 \le 
\sum\limits_{p\le x} \sum\limits_{\substack{q\in \mathcal{P}\\ q|(4p-a_p^2)}} 1 + 
\sum\limits_{p\le x} \sum\limits_{\substack{q\in \mathcal{P}\\ q|(4p-b_p^2)}} 1
$$
and 
$$
\sum\limits_{p\le x} \left(\sum\limits_{\substack{q\in \mathcal{P}\\ q|(4p-a_p^2)(4p-b_p^2)}} 1\right)^2 \le 
2\sum\limits_{p\le x} \left(\sum\limits_{\substack{q\in \mathcal{P}\\ q|(4p-a_p^2)}} 1\right)^2 + 
2\sum\limits_{p\le x} \left(\sum\limits_{\substack{q\in \mathcal{P}\\ q|(4p-b_p^2)}} 1\right)^2.
$$
The bounds $(25)$ and $(26)$ from \cite{CFM} imply that
$$
\sum\limits_{p\le x} \sum\limits_{\substack{q\in \mathcal{P}\\ q|D(4p-a_p^2)}} 1 + 
\sum\limits_{p\le x} \sum\limits_{\substack{q\in \mathcal{P}\\ q|D(4p-b_p^2)}} 1 \ll \frac{x}{\log x}
$$
and 
$$
2\sum\limits_{p\le x} \left(\sum\limits_{\substack{q\in \mathcal{P}\\ q|D(4p-a_p^2)}} 1\right)^2 + 
2\sum\limits_{p\le x} \left(\sum\limits_{\substack{q\in \mathcal{P}\\ q|D(4p-b_p^2)}} 1\right)^2 \ll 
\frac{x}{\log x}
$$
for any non-square positive integer $D$, provided that $z\le x^{1/2}$. In particular, this holds for $D=2$. Since all primes $q$ in the above
two bounds are odd if $z>2$, these bounds remain true if we set $D=1$ in this case.  Using \eqref{primenumber} in addition, it follows that
\begin{equation} \label{newbound}
S(E_1,E_2;x)\ll \frac{x/\log x}{z/\log z}+\frac{1}{P^{2}} \cdot \sum\limits_{\substack{q_1,q_2\in \mathcal{P}\\ q_1 \neq q_2}} \left|\sum\limits_{\substack{p\le x\\
 p \ndivides | q_1 q_2 N}} \left(\frac{(4p-a_p^2)(4p-b_p^2)}{q_1q_2}\right)\right|
\end{equation}
if $z>2$.

Again we shall choose $z$ depending on $x$ so that $z$ exceeds every given real number if $x$ is 
large enough and hence
\eqref{iso} holds for $z/2<q_1,q_2\le z$. In the following, we assume that we are in this situation.
The second term on the right-hand side of \eqref{newbound} is now treated
as in the previous sections, but here we apply Theorem \ref{uncond} in place of Theorem 
\ref{chebo1}.  Using the estimates in Lemma \ref{ndL}, the equation \eqref{sum} and $\# \tilde{C}\le \# C$, we arrive at 
\begin{equation*}
 S(E_1,E_2;x)\ll \frac{x/\log x}{z/\log z} + \frac{x}{\log x}\cdot 
\exp\left(-c_1\cdot \frac{\log x}{z^{16}}\right)+ 
z^{14} x\exp\left(-A\sqrt{\frac{\log x}{z^{14}}}\right)
 \end{equation*}
for some constants $c_1>0$ and $A>0$, provided that
 \begin{equation} \label{cond2}
 \log x\ge c_2z^{42}(\log z)^2
\end{equation}
 for some constant $c_2>0$. This condition comes from condition \eqref{cond1} and Lemma \ref{ndL}. Choosing
\begin{equation} \label{zc}
z:=c_3(\log x)^{1/42}(\log\log x)^{-1/21}
\end{equation}
for some constant $c_3>0$ which is small enough so that \eqref{zc} is 
consistent with \eqref{cond2}, we obtain Theorem \ref{thmunconditional}.

\begin{remark}
We note that the choice of $z$ in \eqref{zc} would contradict condition \eqref{uppercond} if we had used
Lemma \ref{squaresieve} instead of Lemma \ref{squaresieve2}. This explains why we work with version 2
of the square sieve in this situation.
\end{remark}

\begin{acknowledgement}
We thank the referees for their valuable suggestions and for pointing us to the recent preprint by Akbary and Park \cite{AP}. 
We also thank A. Akbary for a careful reading of the introductory part of this paper and for several helpful comments. 
\end{acknowledgement}

\end{document}